\documentclass[12pt]{amsart}    
\usepackage{amscd}

\def\cal#1{\mathcal{#1}}

\def\NZQ{\Bbb}               
\def\NN{{\NZQ N}}
\def\QQ{{\NZQ Q}}

\def\PP{{\NZQ P}}

\def\PP{{\NZQ P}}

\def\frk{\frak}               

\def\mm{{\frk m}}

\def\opn#1#2{\def#1{\operatorname{#2}}} 
\opn\chara{char}
\opn\length{\ell}
\opn\pd{pd}
\opn\rk{rk}
\opn\projdim{proj\,dim}
\opn\rank{rank}
\opn\depth{depth}
\opn\grade{grade}
\opn\height{ht}
\opn\embdim{emb\,dim}
\opn\codim{codim}
\def\OO{\mathcal{O}}
\opn\Tr{Tr}
\opn\bigrank{big\,rank}
\opn\superheight{superheight}\opn\lcm{lcm}
\opn\trdeg{tr\,deg}%
\opn\reg{reg}
\opn\lreg{lreg}
\opn\div{div}
\opn\Div{Div}
\opn\WDiv{WDiv}
\opn\cl{cl}
\opn\Cl{Cl}
\opn\Spec{Spec}
\opn\Supp{Supp}
\opn\supp{supp}
\opn\Sing{Sing}
\opn\Ass{Ass}
\opn\Assh{Assh}
\opn\Min{Min}
\opn\Reg{Reg}
\opn\Ann{Ann}
\opn\Rad{Rad}
\opn\Soc{Soc}
\opn\Socle{Socle}
\opn\Ker{Ker}
\opn\Coker{Coker}
\opn\Im{Im}
\opn\Hom{Hom}
\opn\Mor{Mor}
\opn\Tor{Tor}
\opn\Ext{Ext}
\opn\End{End}
\opn\Aut{Aut}
\opn\id{id}

\opn\nat{nat}
\opn\pff{pf}
\opn\Pf{Pf}
\opn\GL{GL}
\opn\SL{SL}
\opn\mod{mod}
\opn\ord{ord}
\opn\Proj{Proj}
\opn\aff{aff}
\opn\con{conv}
\opn\relint{relint}
\opn\st{st}
\opn\lk{lk}
\opn\cn{cn}
\opn\core{core}
\opn\vol{vol}
\opn\link{link}
\opn\star{star}
\opn\gr{gr}

\def\pot#1#2{#1[\kern-0.28ex[#2]\kern-0.28ex]}

\opn\dirlim{\underrightarrow{\lim}}
\opn\inivlim{\underleftarrow{\lim}}
\let\tensor=\otimes
\let\iso=\cong

\let\to=\rightarrow

\let\To=\longrightarrow

\def\Implies{\ifmmode\Longrightarrow \else
     \unskip${}\Longrightarrow{}$\ignorespaces\fi}
\def\implies{\ifmmode\Rightarrow \else
     \unskip${}\Rightarrow{}$\ignorespaces\fi}
\def\iff{\ifmmode\Longleftrightarrow \else
     \unskip${}\Longleftrightarrow{}$\ignorespaces\fi}

\let\:=\colon
\opn\H{H}
\opn\Pic{Pic}

\newtheorem{Theorem}{Theorem}
\newtheorem{Corollary}[Theorem]{Corollary}
\newtheorem{Lemma}[Theorem]{Lemma}
\newtheorem{Proposition}[Theorem]{Proposition}
\newtheorem{Remark}[Theorem]{Remark}

\newtheorem{Definition}[Theorem]{Definition}

\let\epsilon\varepsilon
\def\OO{{\cal O}} 

\opn\inii{in}
\opn\inim{inm}
\opn\set{set}
\def\pnt{{\raise0.5mm\hbox{\large\bf.}}}

\begin{document}

\title{On Kodaira type vanishing
for Calabi-Yau threefolds in positive characteristic}

\author{Yukihide Takayama}
\address{Yukihide Takayama, Department of Mathematical
Sciences, Ritsumeikan University, 
1-1-1 Nojihigashi, Kusatsu, Shiga 525-8577, Japan}
\email{takayama@se.ritsumei.ac.jp}

\def\Coh#1#2{H_{\mm}^{#1}(#2)}
\def\eCoh#1#2#3{H_{#1}^{#2}(#3)}

\newcommand{\AppTh}{Theorem~\ref{approxtheorem} }
\def\da{\downarrow}
\newcommand{\ua}{\uparrow}
\newcommand{\namedto}[1]{\buildrel\mbox{$#1$}\over\rightarrow}
\newcommand{\bdel}{\bar\partial}
\newcommand{\proj}{{\rm proj.}}

\newenvironment{myremark}[1]{{\bf Note:\ } \dotfill\\ \it{#1}}{\\ \dotfill
{\bf Note end.}}
\newcommand{\transdeg}[2]{{\rm trans. deg}_{#1}(#2)}
\newcommand{\mSpec}[1]{{\rm m\hbox{-}Spec}(#1)}

\newcommand{\tbf}{{{\Large To Be Filled!!}}}

\pagestyle{plain}
\maketitle

\def\gCoh#1#2#3{H_{#1}^{#2}\left(#3\right)}
\def\subsetneq{\raisebox{.6ex}{{\small $\; \underset{\ne}{\subset}\; $}}}
\opn\Exc{Exc}

\def\HHom{{\cal Hom}}

\begin{abstract}
We consider Calabi-Yau threefolds $X$ over an algebraically closed
field $k$ of characteristic $p>0$ that are not liftable to 
characteristic $0$ or liftable ones with $p=2$. 
It is unknown whether Kodaira vanishing holds 
for these varieties. In this paper, we give a lower bound 
of $h^1(X, L^{-1})=\dim_k H^1(X, L^{-1})$ if $L$ is an ample divisor 
with $H^1(X, L^{-1})\ne{0}$. Moreover, we  show that a Kodaira 
type vanishing holds if $X$ is a Schr\"oer variety \cite{Schr03}
or a Schoen variety \cite{Scho},
which extends the similar result 
given in \cite{TakCALG} for the Hirokado variety \cite{Hi99}.
We show that such kind of vanishing holds for Calabi-Yau 
manifold whose Picard variety has no $p$-torsion.
Also we show that a modified Raynaud-Mukai construction
\cite{TakPAMS} does not produce any counter-example to Kodaira
 vanishing.\\
2000 Mathematics Subject Classification: 14F17, 
14J32, 
14G17, 
\end{abstract}

\section{Introduction}

It is well known that K3 surfaces over an algebraically closed field
$k$ of characteristic $p>0$ are liftable to 
characteristic $0$ \cite{D81}, so that Kodaira vanishing theorem holds
by the celebrated result of
 Raynaud-Deligne-Illusie \cite{DI}.  However, some of
Calabi-Yau threefolds are non-liftable to characteristic~$0$ \cite{Hi99,
  Hi07, Hi08, Schr03, Scho, CvS, CS} and it is still open whether
Kodaira vanishing holds for these varieties. Moreover, the theorem by
Raynaud-Deligne-Illusie is not applicable in the case of $p=2$, even
for liftable Calabi-Yau threefolds.

Kodaira vanishing of K3 surfaces has much simpler
proof by S.~Mukai \cite{MuKino}. Although Mukai's 
proof cannot be extended to dimension larger than $2$,
we obtain, under a certain condition, a lower bound evaluation
$h^1(X, \OO_X(-D)) \geq \displaystyle{\frac{1}{6}}D^3$
for an ample divisor $D$
with $H^1(X, \OO_X(-D))\ne{0}$ if there is such $D$ (see Theorem~\ref{main2}). 
To prove this result,
we need to extend Miyaoka's pseudo-effectivity theorem of second Chern class
in characteristic~$0$ \cite{Miya76} to positive characteristic. 
This extension should be of independent interest  (see Theorem~\ref{main1}).

On the other hand, by the Mukai's idea we obtain a sufficient 
condition for a part of Kodaira vanishing of varieties $X$
of $\dim X\geq 2$ with $H^1(X,\OO_X)=0$ such as Calabi-Yau varieties.
Namely, for an ample divisor $D$
we have  $H^1(X, \OO_X(-D))=0$ if $H^0(X, \OO_X(D))\ne{0}$.
It turns out that 
by the result given in \cite{TakCALG} 
this condition can be weaken to $H^0(X, \OO_X(pD))\ne{0}$
when $X$ has no global 1-forms. One of such varieties
is the Hirokado variety \cite{Hi99} as observed  in \cite{TakCALG}.
In this paper, we show that a larger class of the varieties 
satisfies this condition such as 
Schr\"oer varieties
\cite{Schr03} and Schoen varieties 
\cite{Scho} (see Corollary~\ref{schrvar}).

By a Calabi-Yau threefold $X$, we mean a smooth projective 
variety with $K_X=0$ and $H^i(X,\OO_X)=0$ for $i=1,2$. 
For simplicity, we will refer to the situation
$H^1(X, \OO_X(-D))=0$ (or $\ne{0}$) as 
{$H^1$-Kodaira (non-)vanishing} in this paper.
In section~\ref{section:1}, we give an extension 
of Miyaoka's pseudo-effectivity theorem. Its application to 
Calabi-Yau threefolds with $H^1$-Kodaira non-vanishing will be presented
in section~\ref{section:2}. 
In section~\ref{section:3}, we show that 
$H^1$-Kodaira vanishing holds to some extent 
for Schr\"oer varieties and Schoen varieties
(Corollary~\ref{schrvar}).
We give a simple observation that such kind of 
vanishing always holds for a Calabi-Yau varietiy whose 
Picard group has no $p$-torsion (Theorem~\ref{noptorsion}).

We also show that counter-examples of Kodaira vanishing 
cannot be constructed with the modified Raynaud-Mukai construction
by giving an answer to the open problem given in \cite{TakPAMS}.

The author deeply thanks Adrian Langer for pointing out 
a serious error in the early version of Theorem~\ref{main1}.

\section{Pseudo-effectivity of $c_2$ for varieties with $K_X=0$
\label{section:1}}

The aim of this section is to prove an extension of 
Miyaoka's pseudo-effectivity theorem of $c_2$
(Theorem~6.1~\cite{Miya76}, see also Lecture~III of \cite{DMV97}) 
to positive characteristic.
In this extension, we have two obstructions. 
The first one is Bogomolov-Gieseker inequality and the second one 
is generic semipositivity of cotangent bundle, both of which 
are key ingredients of Miyaoka's original proof and we must assume
characteristic $0$. 
We avoid the former by using Langer's result on strongly semistability of 
vector bundles \cite{Lan04}
and the latter by posing a condition for
the Harder-Narasimhan slope of the (co)tangent bundle,
which is also considered in \cite{Lan04}.

Let $X$ be a smooth projective variety and 
${\frak B} = (H_1,\ldots, H_{n-1})$ with $H_i$, $i=1,\ldots, n-1$,
 ample line bundles on $X$. 
Recall that a torsion free sheaf $E$ on $X$ is 
{\em ${\frak B}$-strongly semistable} if all the Frobenius pull backs 
$(F^e)^*(E)$ $(e\geq 0)$
of $E$ are ${\frak B}$-semistable, i.e., for any proper subsheaf
$F\subset E$ we have $\mu(F)\leq \mu(E)$, 
where  $\mu(E) := \mu_{{\frak B}}(E) = 
\frac{c_1(E)H_1\cdots H_{n-1}}{\rank E}$.

For a strongly semistable sheaf, we have the following 
Bogomolov-Gieseker type inequality.

\begin{Theorem}[Theorem~01 \cite{Lan04}]
\label{LBGinequlaity}
Let $X$ be a smooth projective variety over an algebraically 
closed field $k$ of $\chara(k)=p>0$ 
and $n=\dim X\geq 2$ and $E$ a strongly semistable torsion-free sheaf. 
Then we have 
\begin{equation*}
    2\rank(E) c_2(E) \vert {\frak B}\vert 
   \geq  (\rank(E)-1)c_1^2(E)) \vert {\frak B}\vert 
\end{equation*}
where $\vert {\frak B}\vert = H_1\cdots H_{n-2}$
with ample line bundles $H_i$.
\end{Theorem}

If $E$ is not strongly semistable,  we can take 
the Harder-Narasimhan filtration with regard to ${\frak B}$:
\begin{equation}
\label{HN}
   0=E_0 \subset E_1\subset\cdots\subset E_m =E\qquad G_i:= E_i/E_{i-1} \; (i=1,\ldots, m)
\end{equation}
then the components $G_i$ are semistable but not always strongly semistable.
We say that $E$ is {\em fdHN} (finite determinacy of the Harder-Narasimhan filtration)
if there exists $k_0\in\NN$ such that all the components of 
the filtration of the $k_0$-th Frobenius pullback
$(F^{k_0})^*E$ are strongly semistable.
Note that if $E_\bullet$ is the Harder-Narasimhan filtration of $(F^{k})^*E$
for some $k\geq k_0$, then $F^*(E_\bullet)$ is the 
Harder-Narasimhan filtration of $(F^{k+1})^*E$.

\begin{Theorem}[Theorem~2.7 \cite{Lan04}]
\label{fdHNproperty}
Every torsion-free sheaf is fdHN.
\end{Theorem}

Recall also that in the decomposition $(\ref{HN})$, we have 
$\alpha_1 > \cdots > \alpha_m$ with $\alpha_i:= \mu(G_i)$
and we define $\mu_{min}(E):= \alpha_m$ and $\mu_{max}(E):= \alpha_1$. 
We say that $E$ is {\em generically ${\frak B}$-semi positive }
if $\mu_{\min}(E)\geq 0$, which is equivalent to say that 
$\mu_{max}(E^\vee) \leq 0$, i.e. the dual $E^\vee$ is 
{\em generically ${\frak B}$-semi negative} (cf. \cite{Miya76}).
Moreover, we define
\begin{equation*}
     L_{max}(E):= \lim_{k\to\infty}\frac{\mu_{max}((F^k)^*E)}{p^k}
\end{equation*}
and 
\begin{equation*}
     L_{min}(E) := \lim_{k\to\infty}\frac{\mu_{min}((F^k)^*E)}{p^k}.
\end{equation*}
Then we have 
$L_{max}(E)\geq \frac{\mu_{max}((F^k)^*E)}{p^k} \geq \mu_{max}(E)$
and 
$L_{min}(E) \leq \frac{\mu_{min}((F^k)^*E)}{p^k} \leq 
\mu_{min}(E)$ for all $k\geq 0$.
See \cite{Lan04} for 
the detail of $L_{max}$ and $L_{min}$.

Now we come to the main theorem in this section. The proof is 
an easy modification of Miyaoka's proof for $\chara(k)=0$ 
using the results of 
Langer.

\begin{Theorem}
\label{main1}
Let $X$ be a smooth projective variety of dimension $n=\dim X \geq 2$
with $K_X=0$ over an algebraically closed field $k$ of $\chara(k)=p>0$. 
Assume that $L_{min}(T_X)\geq 0$ or $L_{max}(T_X)\leq 0$
for the tangent bundle $T_X$.
Then for any ample divisors $H_1,\ldots, H_{n-2}$ 
we have $c_2(X)H_1\cdots H_{n-2}\geq 0$.
\end{Theorem}

\begin{proof}[Proof of Theorem~\ref{main1}]
Let ${\frak B} = H_1,\ldots, H_{n-2}$ be a sequence of 
$\QQ$-ample divisors and an ample  $\QQ$-divisor $D$.
By Theorem~\ref{fdHNproperty}, there exists $k_0\in\NN$ such that 
$E:= (F^{k_0})^*T_X$ has the Harder-Narasimhan filtration
with regard to $({\frak B}, D)$:
\begin{equation*}
0=E_0\subset E_1\subset \cdots \subset E_s=E
\qquad G_i:= E_i/E_{i-1} \; (i=1,\ldots, s)
\end{equation*}
whose components are all strongly $({\frak B},D)$-semistable.
Then we have $c_1(E) = \sum_i c_1(G_i)$ and 
\begin{eqnarray*}
 2c_2(E)\vert {\frak B}\vert
 & = & \{2\sum_i c_2(G_i) + 2\sum_{i<j} c_1(G_i)c_1(G_j)\}\vert {\frak B}\vert\\
 & = & \{2\sum_i c_2(G_i) + c_1^2(E) - \sum_i c_1^2(G_i)\}\vert {\frak B}\vert
\end{eqnarray*}
where $\vert {\frak B}\vert = H_1\cdots H_{n-2}$.
Now since $c_1(E) = c_1((F^{k_0})^*(T_X)) = (F^{k_0})^*c_1(T_X)
= (F^{k_0})^* c_1(-K_X) = 0$ and $G_i$ are strongly $({\frak B}, D)$-semistable
we obtain by Theorem~\ref{LBGinequlaity}
\begin{eqnarray*}
2c_2(E)\vert {\frak B}\vert
& \geq & \left[\sum_i \frac{r_i-1}{r_i}c_1^2(G_i) - \sum_i c_1^2(G_i)\right]
          \vert {\frak B}\vert
 =   - \sum_i\frac{1}{r_i}c_1^2(G_i)\vert {\frak B}\vert \\
\end{eqnarray*}
where $r_i=\rank G_i$.
Let $h = th_0$ be an ample $\QQ$-Cartier divisor, where $h_0$ 
is ample and $0<t\in \QQ$ is sufficiently small. 
Define $\beta_i$ by $\beta_i h^2\vert {\frak B}\vert
= \frac{c_1(G_i)}{r_i}h\vert {\frak B}\vert$.
Set $\Delta:= \beta_i h \vert{\frak B}\vert 
- \frac{c_1(G_i)}{r_i}\vert{\frak B}\vert$
and $H := h\vert {\frak B}\vert$. 
Then we have $\Delta .H=0$.
Since 
we can assume that $\vert{\frak B}\vert$ is an irreducible 
complete surface, we have 
$\Delta^2.H^2 \leq (\Delta.H)^2 = 0$ 
by a variant of 
Hodge index theorem (Theorem~1.6.1~\cite{LaI04}, Excercise~V.1.9 \cite{H})
and thus $\Delta^2 \leq 0$ since $H$ is ample.
Hence 
\begin{equation*}
c_1^2(G_i)\vert {\frak B}\vert = r_i^2(\beta_ih- \Delta)^2\vert{\frak B}\vert
           \leq r_i^2\beta_i^2h^2 \vert{\frak B}\vert
\end{equation*}
so that 
$2 c_2(E)\vert {\frak B}\vert
\geq  - \sum_ir_i\beta_i^2 h^2 \vert {\frak B}\vert$.
Now since we have $L_{min}(T_X)\geq 0$ or $L_{max}(T_X)\leq 0$, 
we have $\beta_i\geq 0$ for all $i$ or $\beta_i \leq 0$ 
for all $i$. Thus we have the inequality
$- \sum_ir_i\beta_i^2 h^2 \vert {\frak B}\vert
\geq -(\sum_ir_i\beta_i)^2 h^2 \vert {\frak B}\vert$
and we have $(\sum_ir_i\beta_i)h^2\vert{\frak B}\vert = \sum_i c_1(G_i)h\vert
{\frak B}\vert = c_1(E)h\vert{\frak B}\vert = 0$. 
Hence we have $c_2(E)\vert {\frak B}\vert\geq 0$
and then
\begin{equation*}c_2(E)\vert {\frak B}\vert
= c_2((F^{k_0})^*T_X)\vert {\frak B}\vert
= (F^{k_0})^*c_2(T_X)\vert {\frak B}\vert
= p^{k_0}c_2(T_X)\vert {\frak B}\vert \geq 0
\end{equation*}
and thus $c_2(X)\vert {\frak B}\vert \geq 0$ as required.
\end{proof}

We note that the condition ``$L_{min}(T_X)\geq 0$ or $L_{max}(T_X)\leq 0$''
is equivalent to ``$L_{max}(\Omega_X^1)\leq 0$ 
or $L_{min}(\Omega_X^1)\geq 0$''. We now consider what this condition
means. We have 

\begin{Proposition}[Cor.~6.2 and Cor.~6.3~\cite{Lan04}]
We have 
\begin{equation*}
\max(L_{max}(\Omega_X^1) - \mu_{max}(\Omega_X^1),\;
\mu_{min}(\Omega_X^1) - L_{min}(\Omega_X^1))
\leq \frac{n-1}{p}\max(L_{max}(\Omega_X^1),\; 0).
\end{equation*}
In particular, we have 
\begin{equation*}
   L_{max}(\Omega_X^1) \leq \frac{p}{p+1-n}\mu_{max}(\Omega_X^1)
\end{equation*}
if $L_{max}(\Omega_X^1)\geq 0$.
\end{Proposition}
Thus in the case of $L_{min}(T_X)\geq 0$, we have 
$L_{max}(\Omega_X^1) = \mu_{max}(\Omega_X^1)$ and 
$L_{min}(\Omega_X^1) = \mu_{min}(\Omega_X^1)$, which implies 
that $\Omega_X^1$ is strongly semistable.
On the other hand, if $L_{max}(T_X)\leq 0$, we have the constraint 
that $\frac{p}{p+1-n}\mu_{max}(\Omega_X)\geq 0$. Thus 
it suffices to have either (i) $p\geq \dim X$ and $\mu_{max}(\Omega_X^1)\geq 0$,
or (ii) $\dim X> p$ and $\mu_{max}(\Omega_X)\leq 0$.

\section{Calabi-Yau threefolds with Kodaira non-vanishing
\label{section:2}}
In this section, we consider an application of 
Theorem~\ref{main1}. We do not know whether 
Kodaira vanishing holds in general for Calabi-Yau threefolds.
But if there is a counter-example to Kodaira vanishing,
then the dimension of its non-vanishing cohomology has a 
certain lower bound, which we will show in this section.

The following corollary is a direct consequence 
of Theorem~\ref{main1}.
\begin{Corollary}
\label{cortomain1}
Let $X$ be a Calabi-Yau threefold with 
$L_{min}(T_X)\geq 0$ or $L_{max}(T_X)\leq 0$ and $D$ ample divisor.
Then $c_2(X).D\geq 0$.
\end{Corollary}

The following key lemma is based on Mukai's idea presented
in  \cite{MuKino}, where he shows Kodaira vanishing 
for K3 surfaces and Enriques surfaces.

\begin{Lemma}
\label{mukaiLemma}
Let $X$ be a normal projective variety of $\dim X\geq 2$
over an algebraically closed field $k$ of $\chara(k)=p>0$
with $H^1(X,\OO_X)=0$ and $D$ an ample divisor. If $H^1(X, \OO_X(-D))\ne{0}$ 
then $H^0(X,\OO_X(D))=0$.
\end{Lemma}
\begin{proof}
Since $X$ is normal and $\dim X\geq 2$, we have 
$H^1(X, \OO_X(-\ell D))=0$ for $\ell\gg 0$
by lemma of Enriques-Severi-Zariski (Corollary~III.7.8~\cite{H}). 
Thus by considering 
the sequence 
\begin{equation*}
H^1(X, \OO_X(-D))\overset{F^*}{\To} H^1(X,\OO_X(-pD)) 
\overset{F^*}{\To} H^1(X,\OO_X(-p^2D))\To\cdots
\end{equation*}
where $F^*$ is the map induced by the absolute Frobenius morphism,
we can find $\nu\in\NN$ such that 
$0\ne H^1(X, \OO_X(-p^\nu D))\overset{F^*}{\To}H^1(X,\OO_X(-p^{\nu+1}D))=0$.
Now set $\tilde{D} := p^\nu D$, which is also ample and 
$H^1(X, \OO_X(-\tilde{D}))
\ne{0}$, and 
\begin{equation*}
 F^* : H^1(X, \OO_X(-\tilde{D})) \To H^1(X, \OO_X(-p\tilde{D}))(=0)
\end{equation*}
is not injective. 
Now we define $B_X(-\tilde{D})$ by the exact sequence
\begin{equation*}
  0\To \OO_X(-D) \overset{F^*}{\To} \OO_X(-p\tilde{D}) \To B_X(-\tilde{D}) \To 0.
\end{equation*}
From this we obtain the long exact sequence
\begin{equation*}
0 \To H^0(X, B_X(-\tilde{D}))
  \To H^1(X, \OO_X(-\tilde{D})) \overset{F^*}{\To} H^1(X, \OO_X(-p\tilde{D}))
\end{equation*}
and we have $H^0(X, B_X(-\tilde{D}))\ne{0}$.  Now we choose an element
$(0\ne) \eta \in H^0(X, B_X(-\tilde{D}))$.

On the other hand, we define $B_X$ by the short exact sequence
\begin{equation*}
    0\To \OO_X \overset{F^*}{\To} \OO_X \To B_X \To 0
\end{equation*}
so that we have $B_X = \OO_X/\OO_X^p$. From this and $k = H^0(X,\OO_X)$
since $k$ is algebraically closed,
we obtain the
long exact sequence
\begin{equation*}
0\To k \To k \To H^0(X,B_X) \To H^1(X,\OO_X) \overset{F^*}{\To} H^1(X,\OO_X)
\end{equation*}
Thus we have $H^0(B_X) \iso \Ker F^* : H^1(X, \OO_X)\to H^1(X, \OO_X)$.
Now by 
considering the injection
\begin{equation*}
   H^0(X, \OO_X(\tilde{D})) \To H^0(X, B_X) \qquad g\longmapsto g^p\eta
\end{equation*}
we have 
\begin{equation*}
    h^0(X, \OO_X(\tilde{D})) \leq \dim_k(\Ker F^*: H^1(X, \OO_X)\to H^1(X,\OO_X))
\leq h^1(X,\OO_X)=0
\end{equation*}
so that 
$h^0(X, \OO_X(\tilde{D})) = h^0(X, \OO_X(pD))=0$ 
hence $h^0(x, \OO_X(D))=0$ as required.
\end{proof}

Now we prove the main theorem in this section.

\begin{Theorem}
\label{main2}Let $X$ be a Calabi-Yau threefold over an algebraically 
closed field $k$ of $\chara(k)=p>0$. Assume that 
we have $L_{min}(T_X)\geq 0$ or $L_{max}(T_X)\leq 0$ and 
there is 
an ample divisor $D$ such that $H^1(X, \OO_X(-D))\ne{0}$. 
Then we have 
\begin{equation*}
h^1(X, \OO_X(-D))\geq h^{2}(X, \OO_X(-D)) + \frac{1}{6}D^3.
\end{equation*}
In particular, $h^1(X, \OO_X(-D))\geq \frac{1}{6}D^3$.
\end{Theorem}
\begin{proof}
Since $K_X=0$,  we have $\chi(\OO_X)=0$ by Serre duality and then
\begin{equation*}
\chi(\OO_X(D)) = \frac{1}{6}D^3 + \frac{1}{12}D.c_2(X)
\end{equation*}
by Riemann-Roch for threefolds. 
Again by Serre duality together with $K_X=0$, we have 
$\chi(\OO_X(D)) = h^0(X, \OO_X(D)) - h^2(X, \OO_X(-D)) + h^1(X,\OO_X(-D))$.
Thus we have 
\begin{equation*}
h^1(X,\OO_X(-D)) = h^2(X,\OO_X(-D)) - h^0(X,\OO_X(D)) + \frac{1}{6}D^3 + 
\frac{1}{12}D.c_2(X)
\end{equation*}
then we apply  Corollary~\ref{cortomain1} and  Lemma~\ref{mukaiLemma}.
\end{proof}

\section{Kodaira vanishing for Calabi-Yau threefolds
\label{section:3}}

In this section, we consider
the modified Raynaud-Mukai construction introduced 
in \cite{TakPAMS}, which 
was regarded as a candidate to construct 
counter-example to Kodaira vanishing for Calabi-Yau
threefolds.
Also we consider $H^1$-Kodaira vanishing 
under certain conditions. 

\subsection{modified Raynaud-Mukai construction}

In \cite{TakPAMS}, the author considered possibility to 
construct counter-examples to Kodaira vanishing theorem 
using the Raynaud-Mukai construction \cite{Ray,MuKino,Mu11} and
introduced a modified version of the construction.
This construction
produces uniruled varieties and we have 
\begin{Proposition}[Corollary~II.6.3~\cite{Ko96}]
\label{uniruledKV}
Let $X$ be a smooth projective variety over a field $k$ of $\chara(k)=p>0$.
Assume that $X$ is not uniruled. Let $L$ be an ample divisor on $X$
such that $(p-1)L - K_X$ is ample. Then $H^1(X, L^{-1})=0$.
\end{Proposition}
Hence for the  projective varieties $X$ with $K_X=0$,
counter-examples of Kodaira vanishing could exist only when $X$ is uniruled.
It is shown that 
we could construct a counter-examples 
if we consider a modified version
of the construction assuming a smooth surface $X$ such that 
(i) $\chara(k)=p=2$ and $3K_X$ is ample with $H^1(X, \OO_X(-3K_X))\ne{0}$,
or (ii) $\chara(k)=p=3$ and $2K_X$ is ample with $H^1(X, \OO_X(-2K_X))\ne{0}$.
(cf. Corollary~3.4~\cite{TakPAMS}).

However, it turned out that such a surface does not exist,
which is a direct consequence of the following facts.

\begin{Theorem}[cf. \cite{Eke88}]
Let $X$ be a smooth minimal surface of general type over an algebraically
closed field $k$ of $\chara(k)=p>0$. Then 
for any $\ell >0$ $H^1(X, -\ell K_X)=0$ except 
when $\ell=1$, $p=2$ $\chi(\OO_X)=1$ and $X$ is 
an inseparable double cover of a K3-surface
or a rational surface.
\end{Theorem}

\begin{Lemma} For a smooth projective surface $X$ over an algebraically
closed field $k$ of $\chara(k)\geq 0$,  
$K_X$ being nef implies that 
$X$ is minimal, i.e,. any birational morphism $f: X\to Y$ is an isomorphism.
\end{Lemma}
\begin{proof} 
If $X$ is not minimal, then there exists 
a non-isomorphic birational morphism $f:X\to Y$. We can assume that 
$f$ is the blow-up at a point $P\in Y$. Then by Castelnuovo's contraction
theorem we have $K_X = f^*K_Y + C$ with a $(-1)$-curve $C\subset X$.
Hence $K_X.C = C.C = -1 <0$ and $K_X$ is not nef.
\end{proof}

\subsection{weak $H^1$-Kodaira vanishing}

As an immediate consequence of Lemma~\ref{mukaiLemma} we have 

\begin{Corollary}
\label{weakH1kv}
Let $X$ be a Calabi-Yau threefold over an algebraically closed field
$k$ of $\chara(k)=p>0$. Then we have $H^1(X, \OO_X(-D))=0$
for any ample divisor $D$ with $H^0(X, \OO_X(D))\ne{0}$.
\end{Corollary}

We have a more refined result when
there is no global 1-form.

\begin{Definition}
Let $X$ be a projective variety over an algebraically closed field $k$
of characteristic $p>0$.
We say that {\em weak $H^1$-Kodaira vanishing} holds for $X$ if 
we have $H^1(X, \OO_X(-D))=0$ for every ample divisor 
$D$ with $H^0(X, \OO_X(pD))\ne 0$.
\end{Definition}

\begin{Theorem}[cf. Theorem~9 \cite{TakCALG}]
\label{takCALGmain}
Let $k$ be an algebraically closed field of characteristic $p>0$
and  $X$ a Calabi-Yau threefold  over $k$. 
If $H^0(X,\Omega_X^1)=0$, then weak $H^1$-Kodaira vanishing 
holds for $X$.
\end{Theorem}

In \cite{TakCALG}, we applied Theorem~\ref{takCALGmain} to 
the Hirokado variety  \cite{Hi99}.  Now we consider more examples
which we can apply this theorem.

\begin{Lemma}
\label{fibratio}
Let $\pi:X\to \PP^1$ be a surjective morphism from a Calabi-Yau
threefold whose fibers are all smooth K3 surfaces 
or normal surfaces whose desingularization are smooth K3 surfaces.
Then $H^0(X, \Omega_X^1)=0$.
\end{Lemma}
\begin{proof}
Since $\pi$ is surjective, we have 
the short exact sequence
\begin{equation*}
0\To \pi^*\Omega_{\PP^1}^1 \To \Omega_X^1 \To \Omega_{X/\PP^1}^1\To 0
\end{equation*}
and the associated long exact sequence
\begin{equation*}
  0\To \pi_*\pi^*\Omega_{\PP^1}^1 \To \pi_*\Omega_X^1 \To \pi_*\Omega_{X/\PP^1}^1.
\end{equation*}
We first consider the case that $X$ is a smooth K3 pencil.
Since the fiber of $\pi$ is connected,
we have $\pi_*\OO_X = \OO_{\PP^1}$ and thus 
$\pi_*\pi^*\Omega_{\PP^1}^1 \iso \Omega_{\PP^1}^1$
by projection formula.
On the other hand, 
by Grauert isomorphism we have, for any $y\in \PP^1$, 
$\pi_*\Omega_{X/\PP^1}^1\tensor k(y)
\iso H^0(X_y, (\Omega_{X/\PP^1}^1)_y)
= H^0(X_y, (\Omega_{X_y}^1)$
and since $X_y$ is a smooth K3 surface
it is $0$ by the theorem of Rudakov-Shafarevich \cite{RS76}.
Thus we have $\pi_*\Omega_{X/\PP^1}^1=0$. Consequently, we have 
$\pi_*\Omega_X^1 \iso \Omega_{\PP^1}^1$ 
and $H^0(X, \Omega_X^1) = H^0(\PP^1, \pi_*\Omega_X^1)
= H^0(\PP^1, \Omega_{\PP^1}^1) = 0$ as required.
Now in the case that fibers are normal surfaces whose desingularizations 
are smooth K3, the same result holds since birational map does not
change $H^0(\Omega_X^1)$. 
\end{proof}

Recall that a Schr\"oer variety \cite{Schr03} is a pencil over $\PP^1$ of 
smooth supersingular K3 surfaces. Since a supersingular K3 surface
is unirational, Schr\"oer varieties are uniruled. Hence,
in view of Proposition~\ref{uniruledKV}, 
there is a possibility of $H^1$-Kodaira non-vanishing.
On the other hand, a Schoen variety \cite{Scho} is a desingularization 
of the fiber product of certain type of elliptic surfaces over $\PP^1$.
It is not known whether it is uniruled or not.
In both varieties, $p=2$ or $p=3$.

\begin{Corollary}
\label{schrvar}
Weak $H^1$-Kodaira vanishing holds for 
a Sch\"oer variety or a Schoen variety.
\end{Corollary}
\begin{proof}
Since Schr\"oer varieties are smooth K3 pencil over $\PP^1$, 
we can apply Lemma~\ref{fibratio} and 
Proposition~\ref{takCALGmain}. For Schoen varieties, 
the fiber is not K3 but 
we also have $H^0(X, \Omega_X^1)=0$ (Proposition~11.5~\cite{Scho})
and we can apply Proposition~\ref{takCALGmain}.
\end{proof}

\begin{Remark}{\em
There are other examples of non-liftable Calabi-Yau threefolds.
The examples given in \cite{Hi07, Hi08} with $p=2,3$,
which are desingularization
of fiber products of two quasi-elliptic surfaces, are uniruled.
So there is a possibility of weak $H^1$-Kodaira non-vanishing.
These examples have 
fibrations $\pi: X\To \PP^1$ whose general fibers are normal surfaces
with smooth K3 desingularization. The author does not know whether 
the singular fibers satisfy the condition of Lemma~\ref{fibratio}.

For other non-liftable Calabi-Yau threefolds, we have 
those given in \cite{CvS, CS}.  We do not know by now if 
they are uniruled or without global 1-forms.
Currently, we know at most 
the vanishing theorem as in Corollary~\ref{weakH1kv}.
}
\end{Remark}

Finally, we show another condition for weak $H^1$-Kodaira vanishing.
We use the following result by Oda, Lang-Nygaard \cite{Oda69, LN80}.

\begin{Theorem}[ Theorem~1.1~\cite{LN80}]
\label{oda69}
Let $X$ be a smooth proper variety over a perfect field $k$ of characteristic
$p>0$. Then there is an exact sequence
\begin{equation*} 
0\To H^1(X, \OO_X) \To  DM(\hbox{}_p\underline{Pic}(X))
\overset{\alpha}{\To}  H^0(X, \Omega_X^1)
\end{equation*}
where $DM(\hbox{}_p\underline{Pic}(X))$ is the dual Dieudonn\'{e}
module of the $p$-torsion of the Picard scheme. The image
of 
$\alpha$ is the set of indefinitely closed $1$-forms, 
 i.e., $0=d\omega = dC\omega = dC^2\omega
=\cdots$ for every $\omega\in \Omega_X^1$ where $C$ is the Cartier operator.
\end{Theorem}
\begin{proof}
By Prop.~5.7 and Cor.~5.12~\cite{Oda69}.
\end{proof}

Now we have 

\begin{Theorem}
\label{noptorsion}
Let $X$ be  a Calabi-Yau variety $X$ of 
dimension $\geq 3$ over an algebraically closed field $k$ of 
characteristic $p>0$. Then,
$\Pic(X)$ (or $NS(X)$) has no $p$-torsion  if and only 
if $H^0(X, \Omega_X^1)=0$. Hence, in particular, 
weak $H^1$-Kodaira vanishing holds if $\Pic(X)$ has no
$p$-torsion.
\end{Theorem}
\begin{proof}
Note that  we have $\Pic(X) = NS(X)$ since $H^1(X,\OO_X)=0$.
By Lemma~4.7~\cite{vdGK03}, $\Pic(X)$ has 
no $p$-torsion if  $H^0(X, \Omega_X^1)=0$.
On the other hand, assume that $\Pic(X)$ has no $p$-torsion.
By Proposition~4.1~\cite{vdGK03}, all global
$1$-forms are indefinitely closed. Hence by 
Theorem~\ref{oda69} 
$H^0(X, \Omega_X^1) = \alpha(DM(\hbox{}_p\underline{Pic}(X)))=0$
as required.
\end{proof}

\begin{Remark}
For a Calabi-Yau variety $X$ of dimension $2$, i.e., K3 surface,
it is well known that $\Pic(X)$ has no torsion and $H^0(X, \Omega_X^1)=0$
(see \cite{A74, RS76}).
\end{Remark}

\begin{Corollary}
Let $X$ be a Hirokado variety or Schr\"oer variety.
Then $\Pic(X)$ has no $p$-torsion.
\end{Corollary}



\begin{thebibliography}{99}

\bibitem{A74} E.~Artin. Supersingular K3 surfaces.
Ann.~Sci.~\'{E}cole Norm Sup. (4$^e$ Serie)7, (1974) 543--568.

\bibitem{CvS}
 S.~Cynk and D.~van~Straten, Duco. Small resolutions and
non-liftable Calabi-Yau threefolds. Manuscripta Math. 130 (2009),
no. 2, 233--249.

\bibitem{CS}  
S.~Cynk and M.~Sch\"utt. Non-liftable Calabi-Yau spaces.
Ark. Mat. 50 (2012), no. 1, 23--40.

\bibitem{D81} P.~ Deligne.
Relevement des surfaces $K3$
en caracteristique nulle. 
Lectre Notes in Math.,~868, pp.58-79, 
Springer, Berlin-New York, 1981


\bibitem{DI} 
 P.~Deligne and L.~Illusie.
Relevements modulo $p\sp 2$ et decomposition du complexe de de
Rham. Invent. Math. 89 (1987), no. 2, 247--270. 

\bibitem{LN80} W.~E.~Lang and N.~O.~Nygaard.
A Short Proof of the Rudakov-Safarevic Theorem. 
Math.~Ann.~251(1980), 171--173.

\bibitem{Eke88} T.~Ekedahl. Canonical models of surfaces of general
  type in positive characteristic. Inst. Hautes Etudes
  Sci. Publ. Math. No. 67 (1988), 97--144.

\bibitem{H} R.~Hartshorne. {\em Algebraic Geometry}. GTM~52, Springer, 1977.

\bibitem{Hi99} 
 M.~Hirokado. A non-liftable Calabi-Yau threefold in
characteristic $3$. Tohoku Math. J. (2) 51 (1999), no. 4, 479--487.

\bibitem{Hi07} 
 M.~Hirokado, H.~Ito and N.~Saito.
Calabi-Yau threefolds arising from fiber products of rational quasi-elliptic 
surfaces~I. 
Ark. Mat. 45 (2007), no. 2, 279--296.

\bibitem{Hi08} 
 M.~Hirokado, H.~Ito and N.~Saito.
Calabi-Yau threefolds arising from fiber products of rational 
quasi-elliptic surfaces~II. 
Manuscripta Math. 125 (2008), no. 3, 325--343


\bibitem{Ko96}J.~Kollar. Rational Curves on Algebraic Varieties. 
Ergebnisse der Mathematik und ihrer Grenzgebiete. 3. Folge. Banc~32
Springer-Verlag, Berlin, 1996.

\bibitem{Lan04} A.~Langer. Semistable sheaves in positive
  characteristic. Ann. of Math. (2) 159 (2004), no. 1, 251--276.

\bibitem{LaI04} R.~Lazarsfeld. {\em Positivity in Algebraic Geometry~1}.
Classical Setting: Line Bundles and Linear Series,
Springer, 2004.

\bibitem{Miya76} Y.~Miyaoka. The Chern classes and Kodaira dimension
  of a minimal variety. Algebraic geometry, Sendai, 1985, 449--476,
  Adv. Stud. Pure Math., 10, North-Holland, Amsterdam, 1987

\bibitem{DMV97} Y.~Miyaoka and T.~Peternell.
{\em Geometry of higher-dimensional algebraic varieties}.
 DMV Seminar, 26. Birkh\"{a}user Verlag, Basel, 1997.

\bibitem{MuKino}  S.~Mukai. On counterexamples for Kodaira vanishing theorem
and Yau inequality in positive characteristic. in: Symposium on Algebraic
Geometry, Kinosaki, 1979, pp.9--23 (in Japanese).

\bibitem{Mu11} 
 S.~Mukai. Counterexamples of Kodaira's vanishing and Yau's
inequality in characteristics. Kyoto J. Math. 53 (2013), no. 2, 515--532.

\bibitem{Oda69} T.~Oda. The first de Rham cohomology group and 
Dieudonne\'{e} modules. Ann.~Sci.~\'{E}cole Norm.~Sup. ($4^e$ s\'{e}ries) 2
(1969) 63--135.

\bibitem{Scho} 
  C.~Schoen. Desingularized fiber products of semi-stable
  elliptic surfaces with vanishing third Betti
  number. Compos. Math. 145 (2009), no. 1, 89--111.

\bibitem{Schr03} 
 S.~Schr\"oer.
Some Calabi-Yau threefolds with obstructed deformations over the Witt
vectors. Compos. Math. 140 (2004), no. 6, 1579--1592.

\bibitem{Ray} 
 M.~Raynaud.  Contre-exemple au ``vanishing theorem''
  en caract\'eristique $p>0$. (French) C. P. Ramanujam---a tribute,
  pp. 273--278, Tata Inst. Fund. Res. Studies in Math., 8, Springer,
  Berlin-New York, 1978.

\bibitem{RS76}
A.~N.~Rudakov and I.~R.~Shafarevich.
Inseparable morphisms of algebraic surfaces. 
(Russian) Izv. Akad. Nauk SSSR Ser. Mat. 40 (1976), no. 6, 1269--1307, 1439.

\bibitem{TakPAMS} Y.~Takayama.
Raynaud-Mukai construction and Calabi-Yau threefolds 
in positive characteristic.
Proc.~Amer.~Math.~Soc. 140 (2012), no. 12, 4063--4074. 

\bibitem{TakCALG} Y.~Takayama. Kodaira type vanishing for the Hirokado
variety. Comm. Algebra 42 (2014), 4744--4750.

\bibitem{vdGK03} G.~ van~der~Geer and T.~Katsura.
On the height of Calabi-Yau varieties in positive 
characteristic. Doc. Math. 8 (2003), 97--113 .

\end{thebibliography}
\end{document}